\documentclass[11pt,reqno]{amsart}
\usepackage{amsmath,amsthm,amssymb,amsfonts}

\newtheorem{theorem}{Theorem}[section]
\newtheorem{lemma}[theorem]{Lemma}

\theoremstyle{definition}

\theoremstyle{remark}
\newtheorem{remark}[theorem]{Remark}

\newcommand\bR{\mathbb{R}}

\newcommand\cR{\mathcal{R}}

\newcommand{\mysection}[1]{\section{#1}
 \setcounter{equation}{0}}

\newcommand{\nlimsup}{\operatornamewithlimits{\overline{lim}}}
\newcommand{\nliminf}{\operatornamewithlimits{\underline{lim}}}

\newcommand\loc{\rm loc}
 
\begin{document}

   \title[Regular boundary points]{On nonequivalence of regular boundary points
 for  second-order elliptic operators} 
  \author{N.V. Krylov}
\thanks{The first author was partially supported by
 NSF Grant DMS-1160569  
and by a grant 
from the Simons Foundation (\#330456 to Nicolai Krylov)  }
\email{nkrylov@umn.edu}
\address{127 Vincent Hall, University of Minnesota,
 Minneapolis, MN, 55455}

\author{Timur Yastrzhembskiy}
\email{yastr002@umn.edu}
\address{127 Vincent Hall, University of Minnesota,
 Minneapolis, MN, 55455}



\begin{abstract}
 In this paper we present  examples of 
 nondivergence form second order 
elliptic operators  
 with continuous coefficients such that $L$ has an irregular
 boundary point that is regular for the
 Laplacian. 
 Also for any eigenvalue spread   $<1$ 
 of the matrix of the coefficients
 we provide an example of operator  
 with discontinuous coefficients that has 
    regular boundary points nonequivalent to Laplacian's
   (we give  examples for each direction of nonequivalence). 
    All examples are constructed for each dimension starting with $3$.
\end{abstract}
\maketitle

 \mysection{Introduction}

  Let $d \geq 2$ be an integer and let
 $\bR^d$ be a Euclidean space of points $ x = (x_1, \ldots, x_d)$. 
 We consider operators $L$ of the form 
   \begin{equation}
                                   \label{1.1}
  L = a^{ij} (x) \frac{ \partial^2}{\partial x_i \partial x_j}  
  \end{equation}
  where $a(x) = (a^{ij} (x))$ is a 
   symmetric matrix-valued Borel measurable function on $\bR^d$ such that
    $$
          \theta |\xi|^2 \leq a^{ij}(x) \xi_i \xi_j \leq \theta^{-1} |\xi|^2
    $$
  for some constant $\theta > 0$ and all $x,\xi\in\bR^{d}$.
Let $\kappa_1(x)$ and $\kappa_2(x)$ be the smallest
and the biggest eigenvalues of $a(x)$. We introduce 
a quantity $\kappa$ describing the eigenvalue spread 
 of  $a$
 as follows:
     $$
        \kappa(x) = \frac{\kappa_1 (x)}{\kappa_2(x)}. 
     $$
   Given a domain $ G\subset\bR^{d}$,
under appropriate assumptions on $a$ one can define 
  the notion of   regular points   relative to $L$ and $G$  
(see Section \ref{section 4.14.1} for the definition). 
  We write $L \leftrightarrow \Delta$ if the regular boundary
points of $L$ and $\Delta$ coincide (for any $ G$).
  We write $L \not \to \Delta$ if $L$ has a regular
  boundary point that is irregular for the Laplace operator 
 and  $\Delta \not \to L$ if it is the other way around.
 
 The characterization of regular boundary
points is an important problem in the theory of elliptic equations
 and it has attracted a number of researchers.
  In case $L = \Delta$ this problem was solved by N. Wiener 
and the answer is given in the form of the so-called 
Wiener's test 
 which involves the sum of certain capacities \cite{W}. 
 Observe that sometimes   
 de  la Vall\'ee Poussin's
criterion (see Exercise 7.6.7 and Remark 7.6.8 in \cite{Kr_96}), valid for operators with continuous coefficients
is more useful (cf.~\cite{Kr_66}).
 Wiener's result
  has been generalized in  \cite{B}  for operators \eqref{1.1} 
 with continuous
 coefficients, where Wiener's type criterion
is given for the regularity of points.
  No new results are obtained  as to
 when the regular points for the operator
and the Laplacian coincide.  

 The application of
Wiener's test in general appears to be a nontrivial problem 
(see Section 7.11  in \cite{IM} for some examples) and 
 it is worth mentioning that 
sometimes one can prove that a certain point is regular or
  irregular by more elementary means. 
To prove that a boundary point is 
 regular relative to $\Delta$
  one can use   barriers  (see, for example, 
Section 7.6 in \cite{Kr_96}
 for  H\"older
  continuous $a$). 
 In another direction to prove that a point   
is irregular relative to $\Delta$ 
  one can construct a Dirichlet problem and a (perhaps, generalized) solution 
 that fails to converge to its boundary data at 
 that point. 
  Such an example was constructed  by H. Lebesgue \cite{L} 
 in 1913 before
  the discovery of the famous  Wiener's  test
 by N. Wiener \cite{W} in 1924.  A different
example, having some similarities with the example in \cite{L},
was given  in 1925  by P.S. Urysohn \cite{U} who apparently
was unaware of \cite{L}. 
  These techniques also work for operators \eqref{1.1} if the coefficients
 are regular enough.
  
    In the general case of operators \eqref{1.1}  several authors
 proved that, under certain conditions
 on the coefficients $a$, the regular boundary points for $L$ and 
the Laplacian 
coincide.
   O.A. Oleinik \cite{O}
   showed this under the assumption that $a$  is $C^{3 + \alpha}$. 
  R.M. Herv\'e \cite{H} proved the claim for operators $L$
 with   H\"older   continuous 
 (which she calls Lipschitz continuous)  coefficients. 
   Significant improvement was done in \cite{K2}, where
    it was shown that it suffices to require $a$ to be uniformly Dini continuous. 
  Later a somewhat weaker 
 condition was discovered  by Yu.A. Alkhutov (see
\cite{A}). 
  In  the same paper  he
also gave an example of operator \eqref{1.1}
   with discontinuous coefficients such that $L
\leftrightarrow \Delta$. 

 However, it seems that the uniform Dini condition 
cannot be significantly relaxed. 
  In particular, in \cite{M1} K. Miller constructed examples of $\Delta \not \to L$ 
  for any $\kappa \in [0.5, 1)$ and $L \not \to \Delta$
  for any $\kappa \in (0, 1)$  if $d = 3$. In all examples the coefficients
 are continuous everywhere except the origin.
  E.M. Landis \cite{L1}  gave a more elementary example 
  of  $\Delta \not \to L$ 
   for $d=3$ and any $\kappa < 1$  and the coefficients $a$ discontinuous at the origin.
 However, in other dimensions there was only 
  partial progress. 
  In the same article \cite{M1} K. Miller provided examples 
  of  $L \not
\to \Delta$  
 for any $\kappa$  close to 1 if 
$d=2$
 and examples of $L \not \to \Delta$
  for  any $\kappa  \in
(0, \frac{1}{d-2})$ if $ d \geq 4$. 
 
  It turned out that $L \leftrightarrow \Delta$ also does not hold
 in the class of operators $L$ with  continuous  coefficients.
 In \cite{M2} K. Miller gave a simple example  
  of $L \not \to \Delta$  with continuous $a$ 
 for any $ d \geq 2$.
 To the best of the authors' knowledge, 
ours are the first examples of operators $L$
with  continuous  coefficients such that
$\Delta \not \to L$. We give these examples
for any $d\geq3$. 
  
   In this article we are 
also going to construct  examples
 of  $\Delta \not \to L$ and $L \not \to \Delta$ 
  with coefficients discontinuous at the origin
   for  any $\kappa < 1$   and 
any $d \geq 3$.
      An important feature of this article is that our
  proofs  avoid any advanced analogues of the Wiener's
  criterion  for operators \eqref{1.1}.
  We use   barriers to prove that a point 
is   regular relative to $L$.
  In case $\Delta \not \to L$ we construct a generalized solution of the Dirichlet
  problem that is discontinuous at a single boundary point.

  Our examples were inspired by
 Lebesgue's construction of an irregular point relative to
$\Delta$.
 It has to be pointed out that our examples 
also have some
 similarities
with Landis's
 ones mentioned above. 
  In particular, we essentially use the so-called
s-potentials introduced in \cite{L}
 arriving at
 them by combination of probabilistic and PDE
ideas which are outlined in Section~\ref{section 4.19.1}.

\mysection{Generalities about regular points}
                                               \label{section 4.14.1}

Let $ G$ be a bounded domain in $\bR^{d}$ and $B(G)$ be the 
space of bounded Borel measurable functions.
 Take $\delta\in(0,1)$.
We restrict our attention to the operators \eqref{1.1} 
 with coefficients of class $C^\delta_{\loc}( G)$.
Take  a domain $ G'\subset
\bar  G'\subset G$. Suppose that $ G'
\in C^{2+\delta}$. By the H\"older-space theory of linear elliptic equations, for any $f\in C^{\delta}(\bar G')$, 
  there exists a unique solution
   $u \in C^{2+\delta}( \bar G')$  
of the equation $Lu=-f$ in $ G'$ with zero boundary
condition on $\partial G'$. We set
$$
\cR( G')f=u.
$$

Then, let
$ G_{n}$, $n=1,2,...$, be 
a sequence of subdomains
of class $C^{2+\delta}$
such that $ G_{n}\subset G_{n+1} 
\subset\bar{ G}_{n+1}\subset G$,
$ G=\cup_{n} G_{n}$.  For $f\in
C^{\delta}_{\loc}(G)\cap B( G)$
and $x\in G$ define  
\begin{equation}
                     \label{3.16.2}
\mathcal{R}( G)f(x)=\lim_{n\to\infty}
\mathcal{R}( G_{n})f(x).
\end{equation}
This definition coincides with Definition 7.2.2
of \cite{Kr_96}. We will borrow a few facts as well
from Chapter 7 of that book. It has to be said,
however, that in Chapter 7 of \cite{Kr_96} the coefficients
of $L$ are supposed to be of class $C^{\delta}(\bR^d)$
rather than $C^{\delta}_{\loc}( G)$. The reader will 
easily check that what we are going to use
is  proved word for word in the same way as in \cite{Kr_96} 
using only that the coefficients
of $L$ are in $C^{\delta}_{\loc}( G)$.
Part of what is below can also be obtained
from  Remark 1.5.1   and Section 1.8   in \cite{L2}.

We know from \cite{Kr_96} that the limit in \eqref{3.16.2}
exists, does not depend on the choice of
$ G_{n}$, and defines an operator satisfying 
$|\cR( G)f|\leq \cR( G)|f|\leq N\sup|f|$,
where the constant $N$ is independent of $f$.

Interior estimates show that, for any domains
$ G'
\subset \bar G' \subset   G'' \subset \bar G''
\subset  G$ there exists a constant $N$ such that
$$
\|\cR( G) f\|_{C^{2+\delta}( \bar G' )}\leq
     N\| f\|_{C^{\delta}( \bar G'')}+    N\sup_{  G }|f|.
$$
We also know that $L\cR( G)f=-f$ in $ G$
if $f\in C^{\delta}_{\loc}(G)
\cap B( G)$.

According to Definition 7.4.1   of \cite{Kr_96} a point $p  \in\partial G$
is called  {\em regular\/}
(relative to $L$ and $ G$) if
 \begin{equation}
                                              \label{4.22.1}
\lim_{ G\ni x\to p  }\mathcal{R}( G)1(x)=0.
 \end{equation}
Otherwise  $p$  is an irregular  point. Observe that
if $ G'\subset G$, $p  \in\partial G'\cap
\partial G$, and $p  $ is regular relative to $L, G$,
then it is also regular relative to $L, G'$. This follows from
the fact that $\cR( G')1\leq\cR( G)1$.

It turns out (Theorem 7.4.7 of \cite{Kr_96}) that,
for any $c>0$, a point $p  $ is   regular relative to
 $L, G$ if and only if
it is regular relative to $L, G\cap B_{c}(p  )$,
where we use the notation
$$
B_{r}(p  )=\{x\in\bR^{d}:|x-p  |<r\},\quad B_{r}=B_{r}(0).
$$

 According to Definition 7.3.2 of \cite{Kr_96},
if $g\in C^{2+\delta}(\bar{ G})$ and 
$x\in G$,  
 \begin{equation}
                                                     \label{4.14.5}
\pi( G)g(x):=g(x)+\mathcal{R}( G)(Lg)(x).
\end{equation}
Obviously, $\pi( G)g\in C^{2+\delta}_{\loc}( G)$
and $L\pi( G)g=0$ in $ G$. By the interior estimates 
for $L$-harmonic functions, for any domain $ G'
\subset\bar G'\subset G$ there exists a constant $N$
such that for any $g\in C^{2+\delta}(\bar{ G})$ 
$$
\|\pi( G) g\|_{C^{2+\delta}( \bar G' )}
\leq\nliminf_{n\to\infty}\|\pi( G_{n})
g\|_{C^{2+\delta}( \bar G' )}
\leq N \nliminf_{n\to\infty}\sup_{\partial G_{n}}|g|
\leq N\sup_{\partial G }|g|.
$$
The inequality between the extreme terms allows one in \cite{Kr_96}
to uniquely extend $\pi( G)$ by continuity to $g\in C(\partial G)$
and obtain an operator $\pi( G)$ such that, for any
$g\in C(\partial G)$, we have $\pi( G)g\in C^{2+\delta}_{\loc}( G)$,
$L\pi( G)g=0$ in $ G$, and $0\leq  \pi( G)g\leq 1$ in $ G$
if $0\leq g\leq1$ on $\partial G$.

By Exercise 7.4.12 of \cite{Kr_96}, a point $p  \in\partial G$
is regular if and only if
\begin{equation}
                                                      \label{4.14.4}
\lim_{ G\ni x\to p  }\pi( G)g(x)=g(p  )
 \end{equation}
for any continuous $g$ given on $\partial G$.  
The fact that the regularity implies \eqref{4.14.4}
for $g\in C^{2+\delta}(\bar G)$ follows from  definition
\eqref{4.14.5} and the fact that $|\cR( G)(Lg)|\leq \cR( G) 1\sup|Lg|$.   In the general case one uses uniform approximations
of $g$ by polynomials and the fact that $|\pi( G)g -\pi( G)q |
\leq\sup|g-q|$.
The opposite
implication follows if one applies definition \eqref{4.14.5} to
$g=|x-p|^{2}$ and observes that $Lg\geq 2d \theta$ and 
$\cR(G)(Lg)
\geq   2d \theta \cR( G)1\geq0$. 

According to Theorem 7.6.4 of \cite{Kr_96} 
a point $p  \in\partial G$
is regular if, for an $r>0$ and $ G_{r}:= 
G\cap B_{r}(p)$,
there exists a function $w$ (called a barrier at $p$) such that
$$
Lw\leq0\quad{\rm in}\quad G_{r},
\quad \inf_{x\in G_{r}\setminus G_{\rho}}
w(x)>0  \quad\forall\rho\in(0,r),
\quad \lim_{ G\ni x\to p}w(x)=0.
 $$

We are also going to use the following.

\begin{lemma}
                                        \label{lemma 4.14.1}
Let $p  \in\partial G$
and $u,w\in C^{2+\delta}(\bar G\setminus B_{\varepsilon}(p ))$
for any $\varepsilon>0$. Suppose that $u,w\geq0$, $Lu,Lw  \leq  0$ in $ G$,
$$
 \beta:=\nliminf_{\substack{x\in\partial G \\x\to p  }}u(x)>
\nliminf_{\substack{x\in  G\\x\to p  }}u(x)=: \alpha ,
$$
$$
\lim_{ G\ni x\to p  }w(x)=\infty. 
$$
 
Then $p  $ is an irregular point relative to $L, G$.

\end{lemma}

\begin{proof}
 Take $\gamma\in(\alpha,\beta)$ and
$\varepsilon, \eta >0$. By definition, for
$x\in  G\setminus\bar B_{\varepsilon}(p  )$,  
$$
u(x)+ \eta  w(x)= - \cR( G\setminus\bar B_{\varepsilon}(p  ))(Lu+ \eta 
Lw)+\pi( G\setminus\bar B_{\varepsilon}(p  ))(u+  \eta  w)(x)
$$
$$
\geq \pi( G\setminus\bar B_{\varepsilon}(p  ))[(u+ \eta  w)\wedge\gamma](x).
$$
We set $[(u +  \eta  w)\wedge\gamma] (p  )=\gamma$ and then 
$(u +  \eta  w)\wedge\gamma$ becomes  
a continuous function in $\bar G$. By sending $\varepsilon
\downarrow 0$ and using   Exercise 7.3.5  from Chapter 7 of \cite{Kr_96},
we obtain for $x\in G$ that
$$
u(x)+  \eta  w(x)\geq\pi( G)[(u + \eta  w)\wedge\gamma] (x)
$$
For $y\in\partial G$, $y\ne p  $ define $g( y )=u(y)  \wedge\gamma$
and set $g(p  )=\gamma$. Then $g$ is a continuous function on
$\partial G$ and $ (u+  \eta  w)\wedge\gamma\geq g$
on
$\partial G$. It follows that
 $
u(x)+ \eta  w(x)\geq\pi( G)g(x)
 $
in $ G$ and, since $  \eta >0$ is arbitrary,
$$
u(x)\geq  \pi( G)g(x),\quad
\alpha    \geq \nliminf_{\substack{x\in  G\\x\to p  }}
\pi( G)g(x),\quad g(p  )=
\gamma>   \alpha    \geq \nliminf_{\substack{x\in  G\\x\to p  }}
\pi( G)g(x),
$$
the latter showing that $p  $ is not regular. The lemma
is proved.                                    \end{proof}
\begin{remark} 
                             \label{remark 2.4}
 Here we present an example of 
 an irregular relative to $\Delta$
  boundary point which is due to H. Lebesgue \cite{L}. 
The goal of this presentation is to make the reader familiar
with some techniques used below in more complicated situations.
   In case $d=3$, Lebesgue
 proved the existence of a domain
for which the origin is not 
 regular relative to $\Delta$.
P.S. Urysohn (\cite{U}) showed that this holds
for $ G$ defined below if $r(x_{1})<\exp(-3/x_{1})$.

          For  $ d = 3$, H.~Lebesgue considers
\begin{equation}
                                                  \label{4.14.1}
         u(x_{1},r ) = \int_0^1 
\frac{ t } { [( t - x_{1})^2 + r^2]^{1/2} } \, dt,\quad r=(x_{2}^{2}
+x_{3}^{2})^{1/2}.
\end{equation}
    It is easy to see that $u $ is a harmonic function on $\bR^3$ except the interval 
   $[0, 1]$ on the $x_1$-axis. The argument
of H.~Lebesgue is that the origin lies on the boundary of the domain between the closures of the level sets $u=1/2$ and $u=2$ and is   irregular, since $u(0+,0,0)=1\ne2$, in light of Lemma \ref{lemma 4.14.1} in which
one takes $w=1/|x|$.

 One can extract some more information
about irregular surfaces in this example.
For $0<x_{1}\leq 1/2$ we have
$$
  u(x_1, r) = u_{1}(x_{1},r)+u_{2}(x_{1},r),
$$
where
$$
u_{2}(x_{1},r):=\int_{2x_1}^{1} \frac{ t } { [( t - x_1)^2 + r^2]^{1/2  } }\,dt\to 1
$$
as $(x_{1},r)\to0$ by the dominated convergence theorem
(in the integrand $t-x_{1}\geq(1/2)t$
and the integrand is less than 2), and
$$
   u_{1}(x_{1},r):=      \int_{0}^{2x_1} \frac{ t } { [( t - x_1)^2 + r^2]^{1/2  } }\,dt 
\,   =\,
r \psi(x_{1}/r) ,
$$ 
where 
$$
\psi(y):=y\int_{  - y }^{y}\frac{1}{(s^{2}+1)^{1/2}}\,ds  
$$
and the equality  is
obtained by the change of variables $t=x_{1}+rs$
and the observation that
$s/(s^{2}+1)^{1/2}$ is odd and its integral over a symmetric interval is zero.
  By using  L'Hospital's rule 
one easily checks that 
$ \psi(t) - 2 t \ln t=[\psi(t)/t- 
2 \ln t]/(1/t)\to0$ 
as $t\to \infty$.
 Also $x_{1}\ln x_{1}\to0$ as $x_{1}
\downarrow 0$, so that $x_{1}\ln (x_{1}/r)-x_{1}\ln (1/r)\to0$ and
              \begin{equation} 
                                           \label{2.5} 
 u(x_1, r)  -[1+ 2x_{1}\ln(1/r)]\to0.
            \end{equation} 
 if $x_{1}/r\to\infty$  and $x_{1}\downarrow0$.

   Now take $\varepsilon\in(0,1)$ and define
    \begin{equation} 
                                      \label{2.6}
      r(x_1)  =  e^{-  \varepsilon/x_{1}} ,
\quad \phi(x_{1}):=u(x_{1},r(x_{1})),\quad x_1 > 0.
    \end{equation}
By  \eqref{2.5}  we have 
  $\phi(0+)= 
1+2\varepsilon $.

Next, define $r(0)=0$ and consider
$$
 G=B_{1} \setminus\{x:(x_{2}^{2}
+x_{3}^{2})^{1/2}\leq r(x_{1}),x_{1}\geq0\}.
$$
It turns out that the origin is an irregular point of
$\partial G$, no matter how small $\varepsilon>0$ is.
This follows at once from Lemma \ref{lemma 4.14.1} if
one takes $w=1/|x|$ and notes that
$u(0+,0,0)=1<1+\varepsilon$.

The above argument is simpler than the one in \cite{U}, yielding
a more general result, and shorter than the one in
 L.  Helms's  elaboration of H.~Lebesgue's original argument  (see 4.4.11 in \cite{LH}), based on calculating $u(x_{1},r)$
explicitly.
 
  The surface  
  $$
     \{ x:  (x_{2}^{2}
+x_{3}^{2})^{1/2}=  e^{-\varepsilon/x_{1}}, x_1 \geq 0\}
  $$   
 is called the Lebesgue spine.
    \end{remark}

\begin{remark}
                                        \label{remark 5.23.1}
It turns out that the origin is an irregular point
for the Laplacian even in a 
 smaller  domain
$$
G'=G\cap(-G)=B_{1} \setminus\{x:(x_{2}^{2}
+x_{3}^{2})^{1/2}\leq e^{-\varepsilon/|x_{1}|}\}.
$$
This easily  follows from   Wiener's test and the fact that
the origin is irregular relative to $G$. 
This is also proved in Remark \ref{remark 4.18.1}.
\end{remark}

  \mysection{ Main results}
                                         \label{section 2}

 Let $d\geq3$. We represent points $x\in\bR^{d}$ as $x=(x_{1},x')$,
where $ x' = (x_2, x_3, \ldots, x_d) \in \bR^{d-1}$.
 
   Let $c>0$ be a number and let $r \in C([0,c])$ be a nonnegative
function such that  $r(0) = 0$. We say that the surface 
  \begin{equation}
                                                   \label{4.6.1}
    |x'|  = r(x_1),\quad x_1 \in [0,c],
   \end{equation}
   is regular relative to $L$   if the origin is regular relative to
$L, G$, where
  \begin{equation}
                                                   \label{4.17.1}
  G= B_{c}  \setminus \{x\in\bR^{d}: |x'|  \leq r(x_1), x_1 \in [0,c]\}.
  \end{equation}
 
 Remember that, according to the definitions in Section
\ref{section 4.14.1}, to talk about regular points of
$ G$ we need the coefficient $L$ to belong to
$C^{\delta}_{\loc}( G)$ for some $\delta\in(0,1)$.

 We are going to consider functions 
$ u(x_1, r), x_1 \in \bR, r \in \bR_+=[0,\infty)$. 
When it makes sense, denote  
 $$ 
  u_{x_1} =
 \frac{ \partial u}{\partial x_1}, \quad
 u_r  = \frac{ \partial u}{\partial r}, \quad 
u_{x_1 x_1} = \frac{\partial^2 u}{\partial x_1 \partial x_1}, \quad 
 u_{rr} = \frac{\partial^2 u}{\partial r \partial r}.
$$

  Let $\lambda (x)$ be any positive bounded function on $\bR^{d}$ and 
   $ a_{\lambda} ( x )$ be a $(d-1) \times (d-1)$
   matrix-valued function with the following entries:  
\begin{equation} 
                                  \label{2.2}
    a_{\lambda}^{ij}  ( x ) =   \lambda(x)   \delta^{ij} +
    ( 1 - \lambda (x) ) \frac{x_i x_j}{|x'|^2},\quad i,j =2, \ldots, d, x' \neq 0,
   \end{equation}
  $$
      a_{\lambda}^{ij} (x) = \delta^{ij}, \quad  i,j =2, \ldots, d, x' = 0.
   $$
  Next, we  define the $d\times d$ matrix-valued $\tilde a_{\lambda}$ as follows
  $$
   \tilde  a_{\lambda}^{1j} (x) = \delta^{1j}, \quad
   \tilde a_{\lambda}^{i1} = \delta^{i1}, \quad 
     i, j =1, \ldots, d, \,
    $$
  \begin{equation} 
                                   \label{2.3}
    \tilde a_{\lambda}^{ ij} (x) = a_{\lambda}^{ij}(x), 
\quad i,j = 2, \ldots ,d.
    \end{equation}
  
  Note that, for   $ x',z' \in \bR^{d-1} \setminus \{0\}$ and
$ (z',  x') = 0 $ we have 
   $$
    a_{\lambda} ( x )  x' =   x'
   ,\quad
       a_{\lambda} ( x )  z' =  \lambda (x)   z'.
     $$
  This implies that, if $x'\ne0$, the matrix $ a_{\lambda} ( x)$
    has 
 (at least) one eigenvalue equal to 1
and the remaining $d-2$ eigenvalues are equal to $\lambda(x)$.

  Let $L_{\lambda}$ be the operator \eqref{1.1} with coefficients $\tilde a_{\lambda}$.
  Then it is a uniformly elliptic operator 
   with bounded coefficients.
 Also note that, if $x' \neq 0$, we have
    $
     \kappa (x) =  \lambda (x)  
    $
       for
   $ \lambda (x) \leq 1$ 
    and 
    $
     \kappa (x) = 1/\lambda (x) 
    $
  for
   $ \lambda (x) > 1$.
\begin{remark}
                                              \label{remark 4.12.1}
  Even if $\lambda\in C^{\infty}(\bR^{d})$, the coefficients
of $L_\lambda$ might fail to be infinitely 
differentiable 
 but   they  still belong to
 $C^{\infty}_{\loc}(   \bR^d \setminus   \{x'=0\} )$.
Furthermore, 
if $\lambda\in C^{\infty}_{\loc}(\bR^d \setminus \{x'=0\})$
and $\lambda$ is continuous at the origin and $\lambda( 0)=1$, then 
$a_{\lambda}\in C^{\infty}_{\loc}(\bR^d \setminus \{x'=0\})$
and $a_{\lambda}$ is  continuous at the origin.  
 In particular, it is continuous in 
the closure of $ G'$ 
 defined in  \eqref{4.17.2}  below
for any continuous
 increasing function $r(t), t \in [0,c]$ such that $r(t)>0$
for $t>0$ and $r(0) = 0$.
 
\end{remark}

  Let $u$ be a sufficiently smooth function such that, for a function
$$
v(x_{1},r),\quad x_{1}\in\bR,\quad
r\in\bR_{+}=[0,\infty),
$$ we have
  $ 
   u(x) = v(x_1, |x'|).
  $
 Then by what was said about the eigenvalues
   of the matrix $a_{\lambda} ( x)$, 
for $x'\ne0$ and $r=|x'|$,
    we have
   \begin{equation}
                                       \label{1.2}
  L_{\lambda} u(x )  =  v_{x_1 x_1}  + v_{rr} + 
\lambda(x)\frac{d-2  }{r} v_r. 
    \end{equation}

  Let us state the main results of this article.
The first two theorems show that
$$
\Delta\not \to L.
$$
 
   \begin{theorem} 
                             \label{theorem 2.1}
   For any $\varepsilon \in (0, 1)$  and 
   $\lambda(x)\equiv 1/\varepsilon$, 
there exists a domain of type 
\begin{equation}
                                     \label{4.17.2}
 G'=B_{c}  \setminus\{x:|x'|\leq r(|x_{1}|)\},
\end{equation}
where $r$ is a continuous function on $[0,c]$
with $r(t)>0$ for $t>0$ and $r(0)=0$,
for which the origin is regular for $\Delta$ and
not regular for $L_{\lambda}$. Obviously,
  the coefficients of
$L_{\lambda}$ are infinitely differentiable
in $\bar G' \setminus \{0\}$
 and 
$\kappa  =  \varepsilon $
  away from the $x_1$-axis.
   \end{theorem}

\begin{remark}
                                               \label{remark 4.17.1}
 By what was said right after \eqref{4.22.1},
if the origin is regular relative to $L,G$, where $ G$
is taken from \eqref{4.17.1}, then it is also regular relative 
to $L,G'$ (this fact can be also easily explained in case $L=\Delta$
by using Wiener's test since the complement of $G'$ is larger
than that of $G$).
In the proof of Theorem \ref{theorem 2.1} a smooth function $r$ will be chosen so that
the origin is regular relative to $\Delta, G$.   
It is also regular relative to $\Delta, G'$ by the above.

Observe further that one can
find a smooth function $\zeta(x)$ with values
in $[\varepsilon,1]$ given in $\bR^{d}\setminus\{0\}$
such that it equals $\varepsilon$  on an open set
containing the interval $[-c,0)$ of the $x_{1}$-axis
and $\zeta(x)=1$ in $G'$. 
Then, for  
   $\nu(x) =\zeta(x)/\varepsilon$ and the operator $L_{\nu}$,
the origin is still an irregular point of $\partial G'$,
since $L_{\nu}=L_{\lambda}$ in $ G'$,
which implies (see Section \ref{section 4.14.1}) that
the origin is also an irregular point relative to $L_{\nu},
 G$. Finally,
notice that for the operator $L_{\nu}$ we have $\kappa\geq\varepsilon$,
and the coefficients of $L_{\nu}$ have only one point of
discontinuity in $\bar G$. 

\end{remark}

The coefficients of $L_{\lambda}$ are discontinuous
 in $\bar G'$ 
in Theorem \ref{theorem 2.1}. Here is a stronger result
the proof of which is more involved.
   \begin{theorem} 
                                       \label{theorem 2.3}
    
  There is a continuous function $\lambda(x) > 0$ and a domain $G'$ of type \eqref{4.17.2}, 
      where $r$ is a continuous function on $[0,c]$
  with $r(t)>0$ for $t>0$ and $r(0)=0$, such that
  the origin is regular relative to $\Delta, G'$ and
  irregular relative to $L_{\lambda}, G'$.
 Moreover, 
 the coefficients of
$L_{\lambda}$ are infinitely differentiable
in $\bar G' \setminus \{0\}$ and continuous in $\bar G'$. 
      \end{theorem}
\begin{remark}
                                    \label{remark 4.24.1}
The example which we use to prove Theorem \ref{theorem 2.3}
has many similarities with what was used in the past. In particular,
Urysohn \cite{U} used a function which is the sum of the Newtonian
potentials. Landis in Section 4 of \cite{L1} used the sum of $s$-potentials
to show that for $d=3$, generally,  $\Delta \not\to
L $.  We consider $d\geq 3$ and use
integrals, that are easier to analyze,  instead of the sums.
In Landis's construction the operator
$L $ has discontinuous coefficients.  
Novruzov \cite{Na} considered domains of type \eqref{4.17.2} 
again when $d=3$ and tried to push further  Landis's
construction to get an operator with continuous coefficients
showing that $\Delta \not\to L $.
 He deals with sums of potentials of variable singularity, 
as we do with integrals instead
of the sums. His idea, which we
also adopt here, is to first construct a
function and then find an equation which the function
satisfies. However, numerous
rather inexplicable
erroneous statements in  \cite{Na}, which were never
commented on, corrected,
 or substantiated in later publications, do not allow 
one to accept his claimed statements
as true  facts. Furthermore, in the example
in \cite{Na} the origin is not only
irregular relative to $L$ but also {\em irregular\/}
relative to $\Delta$. 
Still in  \cite{KL}  the authors refer to the example in 
\cite{Na} as   a valid example 
showing that 
$\Delta\not \to L  $  with $L $ having
continuous coefficients.
\end{remark}

Here is a result showing that
$L_{\lambda} \not \to \Delta$.

    \begin{theorem}
                                \label{theorem 2.2}
    For any $\varepsilon \in (0, 1)$ and
$\lambda(x) \equiv \varepsilon$,
  there exists a domain $G'$ of type \eqref{4.17.2},
  where $r$ is a continuous function on $[0,c]$
  with $r(t)>0$ for $t>0$ and $r(0)=0$ 
  such that the origin is regular relative to 
  $L_{\lambda}, G'$ and irregular relative to  $\Delta, G'$. 
Moreover, the coefficients of
$L_{\lambda}$ are infinitely differentiable
in $\bar G' \setminus \{ 0\}$
 and $\kappa  =   \varepsilon $
   away from the  $x_1$-axis. 
 
     \end{theorem}

In comparison with Miller's examples from \cite{M1} and
\cite{M2}  of
$L  \not \to \Delta$, for any $d\geq 3$
(and not only for $d=3$ as in \cite{M1}) 
, we present a domain $ G$
such that the interior of $\bar G $ is $ G$.

  \mysection{Auxiliary results}
                            \label{section 3}
Take and fix   constants $b\in(-\infty,0],c\in(0,\infty)$.

 \begin{lemma}
                              \label{lemma 3.2}
        Let $\mu$ be a continuous function 
on $[b,c]$  and $\beta   $ and $\nu$ be  
   continuous  functions  on $[b, c]\setminus\{0\}$ such that $\nu(t)$ is bounded,
 $\nu(t)>1$ and $\beta(t)>0$ on $[b, c]\setminus\{0\}$, and
   \begin{equation}
                                                         \label{4.9.1}
\int_b^c \beta(t)  \, dt<\infty,
\quad\int_b^c \beta(t)  |t| ^{- \nu(t)} \, dt  =\infty.
\end{equation}
If $x \not\in[b,c]$ or $r>0$ set
$$
k_{1}(x ,r)=\int_{b}^{c}\frac{\mu(t)\beta(t)}{[(t-x )^{2}+r^{2}]
^{\nu(t)/2}}\,dt,
$$
$$
k_{2}(x  ,r)=\int_{b}^{c}\frac{ \beta(t)}{[(t-x )^{2}+r^{2}]
^{\nu(t)/2}}\,dt,
$$
      $$
           \omega (x , r) = \frac{ k_1(x, r)} {k_2 (x, r)} .
     $$
If $x \in[b,c]$  set
$$
           \omega (x ,0) = \mu(x ) .
     $$

      Then $\omega(x , r)$ is a continuous function
on $\bR\times\bR_{+}$.
    \end{lemma}

     \begin{proof}
      Clearly, we only need to show the continuity 
on the segment $[b,c]$ of the $x $-axis. 
Observe that owing to Fatou's lemma, \eqref{4.9.1}, and the fact that
$\nu>1$ and $\beta>0$ on
$[b,c]\setminus\{0\}$, for any
$x_{0}\in[b,c]$ we have
\begin{equation}
                                                         \label{4.9.2}
\lim_{\substack{r\downarrow0,\\x \to x_{0}}}k_{2}(x ,r)=\infty.
\end{equation}
Furthermore, if $\delta(\varepsilon)$ is the modulus of continuity
of $\mu$, then from
$$
k_{1}(x ,r)=\mu(x_{0})k_{2}(x ,r)+
\int_{b}^{c}\frac{(\mu(t)-
\mu(x_{0}))\beta(t)}{[(t-x )^{2}+r^{2}]
^{\nu(t)/2}}\,dt
$$
and the estimates $|\mu(t)-
\mu(x_{0})|\leq
\delta(\varepsilon)$ if $|t-x_{0}|\leq\varepsilon$ and 
$|\mu(t)-\mu(x_{0})|
\leq 2\sup  | \mu | $ we see that
$$
|k_{1}(x ,r)-\mu(x_{0})k_{2}(x,r)|\leq
\delta(\varepsilon)k_{2}(x ,r)+2\sup 
 | \mu | \int_{b}^{c}
\frac{\beta(t)}{\varepsilon^{\nu(t)}}\,dt,
$$
where the last integral is finite for any $\varepsilon>0$
since $\nu$ is bounded. Now our assertion
follows in light of \eqref{4.9.2}. The lemma is proved.
     \end{proof}

\begin{remark}
                                           \label{remark 4.22.1}
Obviously, if $\mu\geq \nu$, where $\nu$ is a number,
then $\omega\geq\nu$. 

\end{remark}

\begin{lemma}
                                     \label{lemma 4.10.1}
Let $\mu $ and $h$ be nonnegative functions
on $[b,c]$ such  $\mu h$ is not identically equal to zero,
$\mu$ is bounded, and
\begin{equation}
                                                  \label{4.12.1}
\int_{b}^{c}  |t| ^{\mu(t)  }h(t)\,dt<\infty.
\end{equation}
For $x\in\bR$, $r\geq0$ introduce
\begin{equation}
                                                  \label{4.11.1}
u(x,r)=\int_{b}^{c}\frac{ |t| ^{\mu(t)}h(t)}
{[(t-x)^{2}+r^{2}]^{\mu(t)/2}}\,dt.
\end{equation}
Then the equation
\begin{equation}
                                                  \label{4.11.2}
u_{xx}(x,r) + u_{rr}(x,r)+\frac{\omega(x,r)}{r}u_{r}=0
\end{equation}
holds in $\bR\times\bR_{+}$ apart from  the interval $[b, c]$ on the 
$x $-axis,
where $\omega(x,r)$
is taken
  from Lemma \ref{lemma 3.2} with 
 $\beta(t)=\mu(t)t^{\mu(t)  }h(t)$, and $\nu(t)=\mu(t)+2$.
\end{lemma}

  \begin{proof}   Observe that, for any constant
$\mu>0$, the function
$$
 \upsilon(x,r) =\frac{1}{(x  ^{2}+r^{2})^{\mu/2}}
$$
satisfies
$$
 \upsilon _{x }=-\mu\frac{x }{(x  ^{2}+r^{2})^{1+\mu/2}},
\quad
 \upsilon _{r}=-\mu\frac{r}{(x  ^{2}+r^{2})^{1+\mu/2}},
$$
$$
 \upsilon _{x x }= -\mu\frac{1}{(x  ^{2}+r^{2})^{1+\mu/2}}
+\mu(2+\mu)\frac{x^{2}}{(x  ^{2}+r^{2})^{2+\mu/2}},
$$
$$
 \upsilon _{rr}= -\mu\frac{1}{(x  ^{2}+r^{2})^{1+\mu/2}}
+\mu(2+\mu)\frac{r^{2} }{(x  ^{2}+r^{2})^{2+\mu/2}},
$$
$$
 \upsilon _{x x }+  \upsilon _{rr}=
-\frac{\mu}{r} \upsilon _{r}.
$$

It follows that
$$
u_{ x x }+u_{ rr}=
\int_{b}^{c}
\frac{\mu ^{2}(t) |t| ^{\mu (t)}h(t)} 
{[(t-x ) ^{2}+r^{2}]^{1+\mu (t)/2}}\,dt,
$$
and, since
$$
u_{ r }=-r\int_{b}^{c}
\frac{\mu  (t) |t| ^{\mu (t)}h(t)} {[(t-x ) ^{2}+r^{2}]^{1+\mu (t)/2}}\,dt,
$$
we have
$$
u_{ x x }+u_{ rr}=-\frac{\omega (x ,r)}{r}u_{ r },
$$
and the lemma is proved. 
  \end{proof}

\begin{remark}
                      \label{remark 4.23.1}
  Assume the notation and conditions of  Lemma \ref{lemma 3.2} and Lemma \ref{lemma 4.10.1}.
 Assume that   $c < 1$, $ b=-c$ and $\mu$ and $h$  are even functions. 
 Also assume that $h(t)$ is  decreasing,
  $\mu(t)$  and $t^{\mu(t)} h(t)$ are increasing functions on the interval $(0, c]$.
  Let 
   $$
           \psi (t,x,r)=\frac{ \beta(t) } { [(t-x)^2 + r^2]^{\nu(t)/2}}. 
  $$

  Recall that $\omega(x,0)=\mu(x)$ on $[b,c]$.
  This implies that  if  a function $\mu$ is not Dini continuous at $0$,
 then so is $\omega(x,r)$ at $(0,0)$.
 We claim that under the same assumption on  $\mu(t)$ the function
 $\omega(x,r)$ is even not Dini continuous in a domain 
  $$
   D = \{(x,r): r^2 + x^2 < c^2,  r > r(|x|)\},
  $$
  where $r(t)$ is any continuous increasing function on 
$[0,c]$  such that   $0<r(t)   < t$
 for $t>0$ and $r(0) = 0$. Moreover, the Dini continuity
fails again at the point $(0,0)$.
 
    To prove the claim observe that 
$$
k_2(x,r)=\int_{-c}^{c} \psi(t,x,r) \, dt
$$
 is an even function in $x$ hence
 we will only consider  the case $x > 0$.
   We have
    \begin{equation}
                                      \label{4.23.1.2}
        |k_1 (x,r) - \mu(0)k_2 (x,r)| \geq 
   (\mu(x/2) - \mu(0)) \left(\int_{-c}^{-x/2} + \int_{x/2}^{c}\right) 
 \psi(t) \, dt.
   \end{equation}
   
   Take any $(x,r) \in D$ and note that $[(t-x)^2 +
r^2]^{\nu(t)}$ is a decreasing function in $t$ on
  the interval $(0, x]$.
  This follows from the fact that  
   $(t-x)^2 + r^2$ is a decreasing function of the argument
$t$ on $[0,x]$ and
  $(t-x)^2 + r^2 < 1$, $\forall t \in [0, x]$,
  and
 the function $\nu$
 is increasing on $(0, c]$.

 This combined with the fact that $\beta$ is  increasing 
 on the interval $(0,c]$ imply that $\psi(t,x,r)$ is an
increasing function  in $t$ on $(0,x]$.
 Therefore, 
   \begin{equation}
                                \label{4.23.1.02}
     \int_{x/2}^{x} \psi(t,x,r) \, dt \geq  \int_{0}^{x/2} \psi(t,x,r) \, dt.
 \end{equation}
  
 Next,  using the substitution 
$t = -s$ and our assumption 
that $\beta$ and $\nu$ are even functions
 we conclude that 
  $$
   \int_{0}^{x/2} \psi(t,x,r) \, dt 
= \int_{-x/2 }^{0}
\frac{\beta(s)}{ [ (s+x)^2 + r^2]^{\nu(s)/2} } \, ds
\geq \int_{-x/2}^{0} \psi(t,x,r) \, dt.
  $$
  By this and \eqref{4.23.1.02} we get that
   $$
    2 \int_{x/2}^{x} \psi(t,x,r) \geq   \int_{-x/2}^{x/2} \psi(t,x,r) \, dt.
   $$
  The latter implies  that
  $$
  3 \left(\int_{-c}^{-x/2} + \int_{x/2}^{c}\right) 
 \psi(t,x,r) \, dt \geq  k_2 (x,r).
  $$
 
   Therefore, the right hand side of \eqref{4.23.1.2} is greater than
     $
       (\mu(x/2) - \mu(0)) k_2(x,r)/3
     $
   and, thus, 
  \begin{equation}
                              \label{4.23.1.3}
   |\omega(x,r) - \omega(0,0)| \geq (\mu(x/2) - \mu(0))/3.
  \end{equation}
 
  If $\omega(x,r)$ were
 Dini continuous in $D$, then 
  there would exist a continuous increasing function 
$\phi(t), t \in [0,c]$,
 such that $\phi(0) = 0$, 
  $$
  \sup_{(x,r) \in D} 
|\omega(x,r) - \omega(0,0)| \leq \phi(\sqrt{x^2 + r^2}), \quad 
  \int_0^c \phi(t)/t \, dt < \infty.
  $$
  Then using the fact that $D$ contains a part of the line $r = x$ and  \eqref{4.23.1.3} we obtain
   $$
     \int_0^c \phi(\sqrt{2}t)/t \, dt  \geq  \int_0^c |\omega(x,x) - \omega(0,0)| \, dx/x \geq 
  \frac{1}{3}\int_0^c \frac{\mu(x/2) - \mu(0)}{x} \, dx =\infty.
  $$
Thus, such $\phi$ does not exist and
the claim is proved.
 
 \end{remark}
 
  \begin{lemma}
                                     \label{lemma 4.11.1}
Let $\mu$ and $h$ be nonnegative functions on $[b,c]$,
let $\mu $   be  nondecreasing  
on $(0,c]$, and 
let   $h$ be    nonincreasing   
on $(0,c]$   such that  $\mu(x)>1$ on $(0,c]$ and
condition \eqref{4.12.1} is satisfied.
Let $r(x)$ be a function on $(0,c]$
such that $(1/2)x\geq r(x)>0$ for small $x$ and
\begin{equation}
                                               \label{5.21.1}
   \nlimsup_{x\downarrow 0} \,(\mu(x)-1)
\ln(r(x)/x)<0,          
\end{equation}
\begin{equation}
                                                      \label{4.11.3}
\lim_{x\downarrow0} r(x)(x/r(x))^{\mu(x)}h(2x)
\frac{1}{\mu(x)-1}=\infty.
\end{equation}
Take $u$ from \eqref{4.11.1}. Then
\begin{equation}
                                                      \label{4.11.5}
\lim_{x\downarrow0}u(x,r(x))=\infty.
\end{equation}
\end{lemma}
\begin{proof} For $0\leq r\leq x$ we have
$$
u(x,r )\geq\int_{x+r}^{2x}\frac{t^{\mu(t)}h(t)}
{[(t-x)^{2}+r^{2}]^{\mu(t)/2}}\,dt
\geq h(2x)\int_{x+r}^{2x}\frac{t^{\mu(t)} }
{[(t-x)^{2}+r^{2}]^{\mu(t)/2}}\,dt,
$$
where in the second inequality we used the fact that
$h$ is a decreasing function.
We make the substitution $t=x+sr$ and observe that
 for $s \geq 0$ 
$$ 
(x/r+s )^{\mu(t)}\geq (x/r)^{\mu(t)}
$$ and, for $ 1\leq s\leq x/r$, (recall that
$r\leq x$)
$$
\frac{x/r}
{(s^{2}+1)^{1/2}}\geq1/\sqrt{2},
\quad \bigg[\frac{x/r}
{(s^{2}+1)^{1/2}}\bigg]^{\mu(t)}\geq \delta
 \bigg[\frac{x/r}
{(s^{2}+1)^{1/2}}\bigg]^{\mu(x)},
$$
where $\delta=2^{-\mu(c)/2}$, and the last inequality holds
since $\mu(t)$ is an increasing function.
 In addition note that, for $s\geq 1$,
 we have 
$$ 
  s^{2}+1\leq 2s^{2} ,\quad(s^{2}+1)^{\mu(x)/2}\leq 
2^{\mu(x)/2}s^{\mu(x )}\leq\delta^{-1}s^{\mu(x )}.
$$

Then for $0\leq r\leq x$ we obtain
$$
u(x,r )\geq rh(2x) \int_{1}^{x/r}\bigg[\frac{x/r}
{(s^{2}+1)^{1/2}}\bigg]^{\mu(t(s))}
\,ds
$$
$$
\geq\delta r(x/r)^{\mu(x)}h(2x)\int_{1}^{x/r}
\frac{1}{(s^{2}+1)^{\mu(x)/2}}\,ds
$$
$$
\geq\delta^2 r(x/r)^{\mu(x)}h(2x)\int_{1}^{x/r}
\frac{1}{s^{\mu(x) }}\,ds
$$    
$$
=\delta^{2} r(x/r)^{\mu(x)}h(2x)\frac{1}{\mu(x)-1}
[1-(r/x)^{\mu(x)-1}].
$$
Now our assertion follows easily.
\end{proof}
 
To prove Theorem \ref{theorem 2.2} we need the following.
 
 \begin{lemma}
                                    \label{lemma 7.1}
  Let  $b = - c$,   $\gamma \in (1, \infty)$,
     $\mu(t) \equiv \mu \in [1, d-2)$,
   $$
 h(t) = |t|^{- 1} |\ln |t||^{-\gamma}
$$ 
  and take  $u(x, r)$  from \eqref{4.11.1}.
 For $x$ small enough  split $u(x, r)$ 
into two parts $u_1 (x, r)$ and $u_2 (x, r)$,
  where the first one is the integral from $-2|x|$ to $2|x|$.   

Then

 (i)  For small enough $x\ne0$ and 
$r\ne0$ we have 
$$
 u_{1}(x, r) \leq N |\ln |x||^{
-\gamma}(|x|/r)^{ \mu-1}
\quad\text{if}\quad \mu>1,
$$ 
$$
 u_{1}(x, r) \leq N  |\ln |x||^{
-\gamma} \ln(|x|/r)
  \quad\text{if}\quad \mu=1\quad\text{and}
\quad 2r\leq|x|,
$$
where $N$ is independent of $x$ and $r$.
  Also $u_{2} (x, r)\to u
 (0, 0) < \infty$ as $(x, r) \to 0$.
  
 (ii)   If $r(x) = r(|x|)=o(|x|)$ as  $x \to 0$  and
 $r(x) $ is differentiable  and $|r'(x)|\leq 1$,
then
    \begin{equation}
                               \label{7.1}
    \lim_{x \to 0} x|\ln x|^{\gamma}
 [u_{2}( x, r(x) ) ]' =: \rho \in (-\infty, 0).
    \end{equation}
\end{lemma}

  \begin{proof}
   Since $u(x,r)$ is an even function in $x$ may 
concentrate on $x > 0$.

  (i) First, let $\mu>1$.
  By the substitution 
$ t=x+sr$ and the facts
  that 
$x^{ \mu - 1} |\ln x|^{-\gamma}$ is an increasing function
for small $x>0$ and $(s^2 + 1)^{-\mu/2  }$ is an integrable function we obtain
  $$
        u_1 (x, r) \leq 
       (2x)^{ \mu - 1} |\ln (2x)|^{-\gamma} 
      \int_{-2x}^{2x} \frac{dt} 
    { [ ( t - x)^2 + r^2 ]^{ \mu/2 } } 
   $$
      $$
 \leq   N  x^{ \mu - 1} |\ln x|^{-\gamma}\, r^{ 1  - \mu} 
       \int_{-\infty}^{\infty} \frac{ ds }
{ [s^2 + 1]^{ \mu/2 } } . 
      $$

 If $\mu=1$, then
$$
u_{1}(x,r)\leq 
        |\ln (2x)|^{-\gamma} 
      \int_{-2x}^{2x} \frac{dt} 
    { [ ( t - x)^2 + r^2 ]^{1/2 } }
$$
$$
\leq N |\ln x|^{-\gamma}\int_{-x/r}^{x/r}\frac{ds} 
{ [s^2 + 1]^{ 1/2 } }
\leq N |\ln x|^{-\gamma}
\bigg(2+\int_{1\leq|s|\leq x/r}
\frac{ds} 
{ |s|  }\bigg)
$$
$$
\leq N|\ln x|^{-\gamma}\ln(x/r).
$$
 
  The dominated convergence theorem 
implies the second claim
 in (i) since, for $|t|\geq 2x$ we have
$|t-x|\geq(1/2)|t|$ and the integrand in the expression
of $u_{2}(x,r)$, written
 as  the integral over $[-c,c]$,
is dominated by $2^{\mu}h(t)$, which is integrable.
 
(ii)  First, observe that,
as is easy to see, we may assume that $c<1$. Next, 
 we have
$$
[ u_{2}(x ,r(x))]'= -
   \frac{2^{ \mu  }\, x^{ \mu - 1} } 
{ [x^{2}+ r^{2} (2x)]^{  \mu/2   }
  |\ln(2 x)|^{\gamma}} 
$$
\begin{equation}
                                       \label{7.2}  
 - \frac{2^{ \mu  }\, x^{ \mu - 1} } 
{ [9x^{2}+ r^{2} (2x)]^{ \mu/2  }
  |\ln(2x)|^{\gamma}} 
+ g(x),
\end{equation}
where 
\begin{equation}
                             \label{7.3}
  g(x) =  \mu\left( \int_{2x}^{c} +  \int_{-c}^{-2x}\right) \frac{ |t|^{ \mu-1} }{ |\ln |t||^{\gamma} }
\frac{ (t - x)- r(x) r' (x)}{
[ (t - x)^{2}+ r^{2} (x)]^{  1+  \mu/2   } }\,dt.
\end{equation}
Obviously, the sum of the first two
 terms on the right in \eqref{7.2}
multiplied by $x |\ln x|^{\gamma}$ 
tends to~$- 2^{\mu} ( 1 + 1/3^{\mu})$.

 Next, we multiply the first term 
  on the right-hand side of \eqref{7.3} by $x |\ln x|^{\gamma}$
and make the
 change of variables $t = xs$
 so that the resulting expression equals
$$
   \mu \int_{2}^{\infty}
\frac{ |\ln x|^{\gamma} }{ |\ln s +\ln x|^{\gamma} }
\frac{ s^{ \mu-1 }[ s -1 - r(x) r'(x)/x] }
{ [ (s - 1)^{2} + (r(x)/x)^2 ]^{  \mu/2  + 1 } }
  I_{ s \leq c/x }\,ds.
$$
On the set $s\in( 2, c/x )$ we have
$\ln s+\ln x\leq\ln c$, $|\ln s+\ln x|\geq
|\ln c|$ ($c<1$), and
$$
\frac{ |\ln x|^{\gamma} }{ |\ln s +\ln x|^{\gamma} }=
\frac{ |\ln s +\ln x -
 \ln s |^{\gamma} } { |\ln s +\ln x|^{\gamma} }
\leq N + N \frac{|\ln s   |^{\gamma}}{|\ln s +\ln x|^{\gamma}}
\leq N + N |\ln s|^{\gamma}.
$$
  One can make a similar argument about the second term on \eqref{7.3}.
Hence by the dominated convergence theorem 
$$
\lim_{x \downarrow 0} x|\ln x|^{\gamma} g(x)
= \mu \int_{2}^{\infty}   \frac{ s^{ \mu - 1 } }{ ( s - 1 )^{ \mu + 1 } } \,ds  
 - \mu \int_{-\infty}^{-2}  \frac{ |s|^{ \mu - 1 } }{ |s - 1|^{ \mu + 1 } } \,ds
  $$
  Observe that
    $$
          \mu \int_{2}^{\infty}   \frac{ s^{ \mu - 1 } }{ ( s -
1 )^{ \mu + 1 } } \,ds  =-
\int_{2}^{\infty}   s^{ \mu - 1}   \,d ( s -
1 )^{ -\mu   }
$$
$$
=   
   2^{ \mu - 1 } +
  ( \mu - 1 ) \int_{2}^{\infty} 
   (1  - s^{-1} )^{-  \mu } s^{-2} \, ds  
   $$
     $$
      =  2^{ \mu - 1 } +  2^{ \mu - 1 } - 1
=2^{\mu}-1.
     $$   
  This combined with the above result concerning the first two terms 
on the right in \eqref{7.2} proves \eqref{7.1}.
The lemma is proved.  
 \end{proof}

\mysection{Proof of Theorems \protect\ref{theorem 2.1} and
\protect\ref{theorem 2.3}}  

{\bf Proof of Theorem \ref{theorem 2.1}}.
{\em Case $d=3$\/}. Take $b=-1$, $c=1$, introduce
$\mu \equiv 1 /\varepsilon$,
$h \equiv 1$, and take
$u(x,r)$ from \eqref{4.11.1}, so that
\begin{equation}
                                                     \label{4.18.4}
u(x,r)=\int_{-1}^{1}\frac{|t|^{1/\varepsilon}}{[(t-x)^{2}+r^{2}]^{1/(2\varepsilon)}}\,dt,
\end{equation}
which is quite similar to \eqref{4.14.1}.

 Then we have \eqref{4.11.2} outside
the interval $[-1,1]$ of the $x$-axis with
$$
\omega(x,r)=\frac{1}{\varepsilon }=\lambda  ,
$$
so that according to \eqref{1.2}
the function $v(x)=u(x_{1},|x'|)$ satisfies $L_{\lambda}v=0$
in $\bR^{3}$ outside
the interval $[-1,1]$ of the $x_{1}$-axis.

Observe that, if  $r(x)=x^{\eta}$, 
for an $\eta>0$, then
the limit in \eqref{4.11.3} equals
$$
\frac{\varepsilon}{1-\varepsilon}\lim_{x\downarrow0}
x^{\eta(1-1/\varepsilon)+1/\varepsilon},
$$
which is infinite if $\eta>1/(1-\varepsilon)$.
 Also in that case \eqref{5.21.1}
is satisfied  
  so
that
 by Lemma \ref{lemma 4.11.1} 
$$
\lim_{ \substack{|x'|=r(x_{1})\\{x_{1}\downarrow
0}}}v(x)=\infty.
$$ and owing to the symmetry of $v$ with respect to $x_{1}$
\begin{equation}
                                                         \label{4.12.7}
\lim_{ \substack{|x'|=r(x_{1})\\{x_{1}\to 0}}}v(x)=\infty.
\end{equation}
  
   However, by the monotone convergence theorem
   \begin{equation}
                          \label{*}
        \lim_{ |x'| \downarrow 0} v (0, x' ) = 2.
  \end{equation}

By Lemma \ref{lemma 4.14.1} with  
   $
       w (x) = 1/|x|, 
 $
for which $L_{\lambda}w\leq 0$, since $\lambda>1$,
 the origin is irregular relative to   
  $L_{\lambda}, G'$,
  where $G'$ is defined by \eqref{4.17.2}.

On the other hand,  by  the criterion due to 
K.   It\^o  and H. McKean 
(see  Section  7.11 in \cite{IM}; in their examples, however, there is an error, see the last line on page 259), 
the origin is a   regular relative to $\Delta$  boundary point  for the surface 
$|x'|=r(x_{1})$, $x_{1}\in(0,1]$,
  because
   $$
     \int_0 \frac{  dx } {x  |\ln [r(x )/ x ]|  } 
          = (\eta - 1)^{-1} \int_0 \frac{dx }{x  |\ln x |} = \infty.
       $$
  Now we get the assertion of the theorem
by the first part of Remark \ref{remark 4.17.1}.

{\em Case $d\geq4$\/}.
Introduce
$$
\mu(t)=(d-2)/\varepsilon , \quad 
h(2t)=|t|^{-1}|\ln |t||^{-\alpha},
$$ 
where $\alpha\in(1,\infty)$, and take $u(x,r)$
from \eqref{4.11.1} with $b=-c$ and $c>0$ so small that $h$ is a decreasing
function on $(0,c]$. Then we have \eqref{4.11.2} outside
the interval $[-c,c]$ of the $x$-axis with
$$
\omega(x,r) = \frac{d-2}{\varepsilon }=\lambda (d-2),
$$
so that, according to \eqref{1.2},
the function $v(x)=u(x_{1},|x'|)$ satisfies $L_{\lambda}v=0$
in $\bR^{d}$ outside
the interval $[-c,c]$ of the $x_{1}$-axis.

For 
$$
      r(x ) = x  |\ln x |^{-\eta} ,
   $$
where $\eta=1/(d-3)$,
the limit in \eqref{4.11.3} equals
$$
 \frac{\varepsilon}{d-2-\varepsilon}
\lim_{x\downarrow0} |\ln x|^{ \eta( (d-2)/\varepsilon-1)-\alpha},
$$
which is infinite if
\begin{equation}
                                                     \label{3.11.6}
(d-3)\alpha< (d-2)/\varepsilon-1 .
\end{equation}
 
The set of $\alpha>1$
satisfying \eqref{3.11.6} is nonempty since $\varepsilon \in (0, 1)$.
 We pick any such $\alpha$.
Also observe that \eqref{5.21.1}
is satisfied.
In that case  \eqref{4.12.7}   holds.

  On the other hand, by the monotone convergence theorem
   $$
        \lim_{    |x'| \downarrow 0 } v (0, x' ) = v(0, 0) = u(0, 0) < \infty.
   $$
  Then by lemma $\ref{lemma 4.14.1}$ with 
 $
    w(x) =  |x|^{ - (d-2) } 
   $
  the origin
  is irregular relative to $L_{ \lambda }, G'$, 
  where $G'$ is defined by \eqref{4.17.2}. 

  At the same time,  by  the criterion 
due to K.  It\^o  and H.~McKean 
(see  Section 7.11 in \cite{IM}) 
the origin is a   regular relative to $\Delta$  boundary point  for the surface 
$|x'|=r(x_{1})$, $x_{1}\in[0,1/2]$,
  because
$$
        \int_0  
   \left(  
 \frac{r(x )}{x }
    \right)^{d-3}   \frac{dx }{x } =-
 \int_0 |\ln x  |^{-1} d|\ln x| = \infty.
   $$
 We finish the proof of the theorem as in the
case of $d=3$. \qed

 \begin{remark}
                        \label{remark 4.1}
  In \cite{L1} E.M. Landis constructed a similar example for $d = 3$ by means
  of the potential
  $$
           \int_{a}^{b} \frac{1}{ [ (t - x_1)^2 + r^2 ]^{ \nu/2 } }\, dt,  
   $$
  where $ 0 \leq a \leq b $, also used in \cite{U}.
    In \cite{L1}, \cite{Na}, and \cite{U} the  solutions are given by   weighted 
  sums of those potentials.
 \end{remark} 
 
{\bf Proof of Theorem  \ref{theorem 2.3}}. {\em Case $d=3$\/}.
    Introduce 
$$
 \mu(t)=1+  (\ln | \ln |t| |)^{-1}, 
\quad h \equiv 1, \quad r(t) =|t|^{1+ |\ln|\ln |t||}.
$$ Notice that for sufficiently small $c>0$ and $b=-c$ the
conditions
of Lemma \ref{lemma 4.10.1} are satisfied and
the conditions of Lemma \ref{lemma 3.2}
are satisfied  with 
\begin{equation}
                                            \label{5.21.2}
  \beta(t)=\mu(t) |t| ^{\mu(t)  }h(t),\quad \nu(t)=\mu(t)+2.
\end{equation}
It follows that $u$ introduced by \eqref{4.11.1}
satisfies equation  \eqref{4.11.2} 
in $\bR\times\bR_{+}$ apart from  the interval $[-c, c]$ on the 
$x $-axis, where
 $\omega$ is
defined in Lemma \ref{lemma 3.2} with the above
 $\mu,\beta,\nu$, 
and this $\omega$ is a continuous function.

Then $v(x)=u(x_{1},|x'|)$ satisfies $L_{\lambda}v=0$
in $\bR^{3}$ apart from  the interval $[ -c , c]$ on the 
$x_{1} $-axis, where $\lambda(x)=\omega(x_{1},|x'|)$
is a continuous function and $\lambda(0)=\omega(0,0)
=\mu(0)
=1$, so that by Remark \ref{remark 4.12.1}
the coefficients of $L_{\lambda}$ are continuous in
the closure of $G ' $,   where  $G'$ is defined by
 \eqref{4.17.2}. 
  In that case, of
course, the coefficients of $L_{\lambda}$  are smooth in
$\bar G'\setminus\{0\}$.
 
Then observe that the expression under the limit
sign  in \eqref{4.11.3} equals $ \ln|\ln x|$
and, thus, \eqref{4.11.3} holds true. Also 
\begin{equation}
                                               \label{4.12.5}
  (\mu(x)-1)
\ln(r(x)/x)= \ln x\to-\infty
\end{equation}
as $x\downarrow0$. Therefore \eqref{5.21.1}
is satisfied as well.
 
Hence by Lemma \ref{lemma 4.11.1} equations \eqref{4.11.5}
and  \eqref{4.12.7}
hold again. Furthermore, since $v(x)$ is even with respect to 
$x_{1}$, we again have \eqref{4.12.7}.
 
 On the other hand, by the monotone convergence theorem
$$
\lim_{ |x'| \downarrow 0 }v(0,x')=
\lim_{r\downarrow0}u(0,r)=2c<\infty.
$$
  By Remark \ref{remark 4.22.1} we have $\omega\geq1$, so
that
$\lambda\geq1$ and $L_{\lambda}w\leq0$ for $w=1/|x|$.   
Therefore,   by  Lemma \ref{lemma 4.14.1}, with $w$ defined
as above,  the origin is not regular relative to
$L_{\lambda}$ , $G'$.

The fact that the origin is regular for $\Delta$,
$G{ '}$ again follows from  Section
7.11 of  \cite{IM}  
 (and Remark \ref{remark 4.17.1}) since 
$$
  \int_0 \frac{  dx } {x  |\ln [r(x )/ x ]|  } 
=\int_{0}\frac{d\ln x}{|\ln x|(
\ln|\ln x|)}=\infty.
$$
This  
  proves the theorem in case $d=3$.

{\em Case $d\geq 4$\/}. This time   take 
a number $\gamma>1$ and define
       \begin{equation}
                               \label{4.12.8}
          \mu(t) = d - 2 + \gamma (d-3)
\frac{\ln\ln|\ln|t||}{\ln|\ln|t||},  
        \end{equation}
     
$$
h(2t) =  |t|^{-1} |\ln |t|
|^{-1}(\ln|\ln |t||)^{-2}  , 
$$

 $$
   r(x) =  |x|\, |\ln |x|
| ^{- 1/(d-3) }(\ln|\ln |x||)^{- 1/(d-3) }.
    $$
Observe that one can choose $c>0$ so small that for $b=-c$ the assertions
made in case $d=3$  before \eqref{4.12.5} are still valid with
thus defined $\mu$, $h$, $r$, and 
$\lambda(x)=(d-2)^{-1}\omega(x_{1},|x'|)$ in place of those $\mu$, $h$, $r$, and $\lambda$ introduced
in case $d=3$.
 Also note that \eqref{5.21.1} holds
because this time $\mu(t)$ does not even go to 1
as $t\downarrow0$.

  Let
    $$
     \phi(t) = \gamma(d-3)\frac{\ln\ln|\ln t|}{\ln|\ln t|}, \quad  g(x)  = \ln|\ln x|.
  $$
    Then observe that, for $x>0$ we have
     \begin{equation}
                                 \label{4.12.9}
        r(x)(x/r(x))^{\mu(x)}h(2x) = |\ln x|^{\phi(x)/(d-3)} (\ln|\ln x|)^{-1 + \phi/(d-3)}   
    \end{equation}
    $$
  = [g(x)]^{\gamma - 1 + 
\gamma (\ln g(x))/g(x)} \to \infty
   $$
  as $x \downarrow 0$ which yields \eqref{4.11.3},
\eqref{4.11.5}, and \eqref{4.12.7}.
 
On the other hand,
by the monotone convergence theorem
$$
\lim_{ |x'|\downarrow0 }v(0,x')=
\lim_{r\downarrow0}u(0,r)=u(0,0)=\int_{-c}^{c}h(t)\,dt<\infty.
$$
This  
 and Lemma \ref{lemma 4.14.1} with $w=1/|x|^{d-2}$,
which satisfies $L_{\lambda}w\leq0$
since $\lambda\geq1$,
shows that the origin is not regular relative to $L_{\lambda} , G'$.

We finish the proof by observing that
$$
        \int_0  
   \left(  
 \frac{r(x )}{x }
    \right)^{d-3}   \frac{dx }{x }=\int_{0}\frac{d\ln x}{|\ln x|
\ln|\ln x|}=\infty,
$$
so that the fact that the origin is regular 
relative to $\Delta,G'$ again follows
from \cite{IM} and  Remark \ref{remark 4.17.1}.
The theorem is proved.   \qed
 
\begin{remark}
                         \label{remark 5.1}
 Observe that in Theorem \ref{theorem 2.3} $\mu$ is not Dini
continuous at $0$ and $r(t) \in (0, t)$,
 $\forall t > 0$, in
both cases. Hence, it follows by  Remark \ref{remark 4.23.1}
(see \eqref{4.23.1.3}) that the modulus of continuity 
$\phi$ of the coefficients of the corresponding differential
operator does not
satisfy the uniform Dini condition in the domain
 $G{ '}$
 from the proof of Theorem \ref{theorem 2.3}. 
However, $\phi$ misses the Dini condition by
quite much. It  even misses, albeit  barely, the condition
\begin{equation}
                                          \label{5.23.5}
\int_{0}\frac{\phi(r)}{r|\ln r|}\,dr<\infty.
\end{equation}

The proof of Theorem \ref{theorem 2.3} shows that if 
 \eqref{5.23.5} (which is weaker than the Dini condition)
is violated, then it may happen that 
  $\Delta \not \to L$.
We could not construct an  example of an operator
$L$ such that \eqref{5.23.5} holds but   $\Delta \not \to L$.
 In this connection a natural question
arises:

Is it true that, if condition
\eqref{5.23.5} is satisfied, then
all points regular relative to $\Delta$
are also regular relative to $L$?
  \end{remark}
 
  \mysection{Proof of Theorem \protect\ref{theorem 2.2}.}
  \textbf{Proof of  Theorem \ref{theorem 2.2}}   
{\em Case $ d \geq 3$ and $\varepsilon<1/(d-2)$\/}. 
Set $\lambda \equiv
\varepsilon$ and observe that  \eqref{1.2} and easy computations
show that 
  the function $v(x) = |x'|^{1-\varepsilon(d-2)}$  satisfies
$L_{\lambda} v=0$ outside of 
 the $ x_{1} $-axis. 

 This $v$ is a barrier
relative to $L_{\lambda}$ at the origin for any domain,
 whose closure intersects
the $ x_{1} $-axis only at the origin.
 It follows that the origin is regular relative to
$L_{\lambda}$ and $G'$  defined by  \eqref{4.17.2}, where
  $r(t)$ is any continuous function on $[0, c)$ for some $c
> 0$ and, moreover,  $r(t) > 0$  for  $t > 0$ and $r(0) =
0$.      We take $r(x_1)
= e^{-\varepsilon/x_1}$ and, by what was said in Remark
\ref{remark 5.23.1},
 the origin is irregular relative to 
$\Delta$ and $G'$. 

{\em Case   $ d \geq 4$ and
$\varepsilon\in(1/(d-2),1)$\/}. 
 Take $u(x, r)$ 
from Lemma \ref{lemma 7.1}
  with $\mu = \varepsilon ( d - 2)$. 
   Then we have \eqref{4.11.2} outside
the interval $[-c,c]$ of the $x_1$-axis with
$
\omega(x,r)=\varepsilon (d-2)=\lambda(d-2),
$
so that according to \eqref{1.2}
the function 
$$
v(x)=u(x_{1},|x'|)
$$
 satisfies $L_{\lambda}v=0$
in $\bR^{d}$ outside
the interval $[-c, c]$ of the $x_{1}$-axis.

 Let $\eta \in (0, 1)$ 
be such that 
$$
\eta(\mu-1)<1,\quad \eta(d-3)>1.
$$
Notice that the set of such $\eta$ is nonempty
because $1/(\mu-1)>1/(d-3)$ since
$d-3>\varepsilon(d-2)-1$. Set
   \begin{equation}
                                \label{5.23.4}
       r(t) = |t|\,|\ln |t|\,|^{ -\eta },
  \end{equation}
  and let $G'$ be  the  domain defined by \eqref{4.17.2}.

   We claim that $v(0) - v(x)$ is a  
 barrier relative to
$L_{\lambda}$, $G'$
(see Section \ref{section 4.14.1} for the definition of 
a barrier).    In order to prove the claim,   
take $u_1(x,r)$ and  $u_2(x,r) $ 
from Lemma \ref{lemma 7.1} and observe that
 \begin{equation} 
                               \label{5.23.2}
    v(0) -  v (x) = u (0, 0) - u_2 (x_1, |x'|) -u_1 (x_1, |x'|).
  \end{equation}
 By   Lemma \ref{lemma 7.1}  (i)  in   $G'$ we have
  $$
          u_1 (x_1, |x'|)\leq
u_1 (x_1, r(x_{1})) \leq N  (|x_1|/r(x_{1}))^{\mu-1} |\ln|x_1||^{-\gamma}
$$
$$
= N|\ln |x_1||^{  \eta(\mu-1)-\gamma} \to 0
  $$
  as $x_1 \to 0$ (recall that $\gamma
\in(1,\infty)$). 
   Also by Lemma
\ref{lemma 7.1} (i)
$u_2 (x_1, |x'|)\to u(0,0)$ as $x\to0$.  
It follows  that $v(x) \in C(\bar G')$
and that to prove the claim it suffices to show
  the following: 
   \begin{equation}
                                \label{5.23.3}
         \nliminf_{x_1 \downarrow 0} \frac{ u(0,0) -   u_2
(x_1, r(x_{1}) ) } { u_1 (x_1, r(x_{1})) } > 1.
    \end{equation}
  
   By  Lemma \ref{lemma 7.1}  (i)  it suffices to prove
\eqref{5.23.3} with $N |\ln |x_1||^{  \eta(\mu-1)-\gamma}$
in place of 
$u_1 (x_1, r(x_{1}))$. After this we use the L'Hospital's
rule   to observe that
$$
\nliminf_{x_1 \downarrow 0} \frac{ u(0,0) -   u_2 (x_1,
r(x_{1}) ) } { |\ln  x_1 |^{ 
\eta(\mu-1)-\gamma} }
=\frac{-1}{\gamma-\eta(\mu-1)}\lim_{x_{1}\downarrow0}
\frac{x_{ 1}|\ln x_{ 1}|^{\gamma}[ u_2
(x_1, r(x_1))]'} {|\ln  x_1 |^{ 
\eta(\mu-1) -1}},
$$ where the last limit does exist and equals $\infty$ in
light of
 Lemma 
\ref{lemma 7.1} (ii).

  Thus, with our claim being proved,
by what is said
before Lemma \ref{lemma 4.14.1}, the origin is a 
regular boundary point relative to $L_{\lambda}, G'$. 
 
   By the criterion from \cite{IM}
(see Section 7.11)  and Remark \ref{remark 4.17.1}
 the origin is an irregular boundary point relative to $\Delta,
G'$ because
   $$
        \int_0   \left( \frac{ r(t)}{t} \right)^{d-3}
\frac{dt}{t}
    = \int_0  \frac{dt} 
 { t |\ln t|^{\eta(d-3)  } } < \infty 
    $$
since by our choice of $\eta$ we have $\eta(d-3)>1$.

{\em Case $d\geq 4$ and $\varepsilon
=1/(d-2)$\/}.
 Take $u(x, r)$ 
from Lemma \ref{lemma 7.1}
  with $\mu =\varepsilon(d-2)=1$ and $\gamma \in (1, 2)$. 
   We have \eqref{4.11.2} outside 
the interval $[-c,c]$ of the $x_1$-axis with
$
\omega(x,r)= 1,
$
so that according to \eqref{1.2}
the function 
$
v(x)=u(x_{1},|x'|)
$
 satisfies $L_{\lambda}v=0$
in $\bR^{d}$ outside
the interval $[-c, c]$ of the $x_{1}$-axis.

  Take any $\eta > 0$ and let a function $r(t)$ and a domain $G'$ be defined by
    \eqref{5.23.4} and \eqref{4.17.2},
  respectively. Take $u_1(x_1, r)$ and $u_2(x_1, r)$
  from Lemma \ref{lemma 7.1} and observe that \eqref{5.23.2} holds. 
 
    By Lemma \ref{lemma 7.1} (i) in $G'$ we have
  $$
   u_1(x_1, |x'|) \leq 
u_1(x_1, r(x_1)) \leq N |\ln|x_1||^{-\gamma} \ln |\ln|x_1||^{\eta}
  $$ 
    \begin{equation} 
                                \label{5.24.1}
         \leq N |\ln|x_1||^{-\gamma/2} \to 0
    \end{equation}
  as $x_1 \to 0$.
  We also have $u_2(x_1, |x'|) \to u(0,0)$ as $x \to 0$ by Lemma \ref{lemma 7.1} (i).
 This and \eqref{5.24.1} imply that $v \in C(\bar G')$. In order to prove that
 $v(0) - v(x)$ is a barrier at the origin relative to $L_{\lambda}, G'$  it suffices to show that \eqref{5.23.3}
 holds. As in the previous case thanks to
 \eqref{5.24.1}
we
  may replace the denominator with  $N |\ln|x_1||^{-\gamma/2}$. 
 Next, by L'Hospital's rule we have
   $$
       \nliminf_{x_1 \downarrow 0} \frac{ u(0,0) -   u_2 (x_1,
r(x_{1}) ) } { |\ln  x_1 |^{-\gamma/2} } = 
  \frac{ - 1} {  \gamma/2} \lim_{x_{1}\downarrow0}
  x_1 |\ln  x_1 |^{\gamma/2 + 1}  [u_2(x_1, r(x_1))]'
   $$
  where the last expression equals $\infty$ due to Lemma \ref{lemma 7.1} (ii) and our choice of $\gamma \in (1, 2)$. 
 This proves \eqref{5.23.3} and hence $v(0) - v(x)$ is a barrier at the origin relative to 
 $L_{\lambda}, G'$. Thus,  the origin is regular relative to $L_{\lambda}, G'$.

 Now we take $\eta > 1/(d-3)$. We know from the previous case that for such $\eta$
  the origin is irregular relative to $\Delta, G'$. This finishes the proof.
  \qed

   \mysection{Comments on some underlying ideas}
                                            \label{section 4.19.1}

Let $a(x)$ be a $d\times d$ symmetric matrix valued Borel measurable function
on $\bR^{d}$ which is bounded and uniformly nondegenerate.
Let $\Omega$ be the set of continuous $\bR^{d}$-valued functions 
$\omega=\omega(t)$
on $[0,\infty)$. For $\omega\in\Omega$ introduce $x_{t}(\omega)=\omega_{t}$
and let $N_{t}$ be the $\sigma$-field of subsets of $\Omega$
generated by the sets $\{\omega:x_{s}(\omega)\in\Gamma\}$ for $s$ running through $[0,t]$
and $\Gamma$ running through the set of Borel subsets of $\bR^{d}$.
By $N_{\infty}$ we denote the $\sigma$-field of subsets of $\Omega$
generated by the sets $\{\omega:x_{s}(\omega)\in\Gamma\}$ for $s<\infty$.
As is known from Theorem 3 of \cite{Kr_73}, for any $x\in\bR^{d}$
there exists a probability measure $P_{x}$ on $\{\Omega,N_{\infty}\}$
such that $X=(x_{t},\infty,N_{t },P_{x})$ 
 is a strong Markov process
and $X=(x_{t},\infty,N_{t+},P_{x})$ is a   Markov process (in the terminology of \cite{Dy_63}) such that for any twice continuously differentiable
function $u(x)$ on $\bR^{d}$ with compact support and any $x\in\bR^{d}$
and $t\geq0$ we have
$$
u(x)=E_{x}u(x_{t})-E_{x}\int_{0}^{t}Lu(x_{t})\,dt.
$$
This property and the Markov property imply that
$$
u(x_{t})-\int_{0}^{t}Lu(x_{t})\,dt
$$
is a martingale relative to $(N_{t},P_{x})$ 
for any $x$ and  this  combined with
the strong Markov property easily shows that, for any bounded
domain $G\subset\bR^{d}$, $u\in C^{2}(\bar G)$, and $x\in G$, we have
\begin{equation}
                                             \label{4.18.1}
u(x)=E_{x}u(x_{\tau_{G}})-E_{x}\int_{0}^{\tau_{G}}Lu(x_{t})\,dt,
\end{equation}
where
$$
\tau_{G}=\inf\{t\geq 0:x_{t}\not\in G\}.
$$
In light of this, naturally, for any bounded domain $G$, Borel bounded $g$
on $\partial G$ and $f$ on $G$ the function
$$
 E_{x}g(x_{\tau_{G}})-E_{x}\int_{0}^{\tau_{G}}f(x_{t})\,dt
$$
is called a probabilistic solution of the equation $Lu=-f$
in $G$ with the Dirichlet boundary condition $g$.
We are interested in the case where $f=0$ and the main issue
for us is whether for a $p\in\partial G$ and any
continuous $g$   it holds that
$$
\lim_{\substack{x\in G\\x\to p}}E_{x}g(x_{\tau_{G}})=g(p).
$$
If it holds indeed, $p$ is called a regular point. 

If the coefficients of $L$ 
are in $C^{\delta}( \bar G )$, $G\in C^{2+\delta}$,
$f\in C^{\delta}(\bar G)$,
and $g\in C^{2+\delta}(\partial G)$, the equation $Lu=-f$ in $G$
with boundary condition $g$ has a unique solution 
$u\in C^{2+\delta}(\bar G)$, which owing to what was said about
\eqref{4.18.1} implies that the probabilistic solution coincides
with $u$. In particular,
$$
\cR(G)f(x)=E_{x}\int_{0}^{\tau_{G}}f(x_{t})\,dt.
$$
This and the fact that $\tau_{G_{n}}\to\tau_{G}$ as $n\to\infty$
if the domains $G_{n}\uparrow G$ implies that the notions
of regular points introduced here and in Section \ref{section 4.14.1}
agree if the coefficients of $L$ are in $C^{\delta}_{\loc}(G)$.

In a subsequent paper we will show that $X$ is a strong Feller
process  (see  Section 13.1 in \cite{Dy_63} for the definition) 
and therefore, owing to \cite{Kr_66}, $p$ is a regular point
if and only if
$$
P_{p}\{\tau'_{G}=0\}=1,
$$
where
$$
\tau'_{G}=\inf\{t>0: x_{t}\not\in G\}.
$$
One knows from Blumenthal's 0-1 law that $P_{p}\{\tau'_{G}=0\}$
is either zero or one. 
\begin{remark}
                                        \label{remark 4.18.1}
Take $G$ and $G'$ from \eqref{4.17.1}
and \eqref{4.17.2}, respectively.
Then the origin is regular with respect to $\Delta,G$
if and only if it is regular relative to $\Delta,G'$.

Indeed, in one way this follows from the fact that $G'\subset
G$ (see Section \ref{section 4.14.1}). In the opposite direction, if the origin is not regular
relative to $\Delta,G$, then $\tau_{G}>0$ ($P_{0}$-a.s.),
owing to symmetry, $\tau_{-G}>0$ ($P_{0}$-a.s.), and hence
$\tau_{G'}=\tau_{G}\wedge\tau_{-G}>0$ ($P_{0}$-a.s.),
so that the origin is irregular relative to $\Delta,G'$.

\end{remark}

Now we can explain the idea behind our examples. 
Denote by $\Delta_{d}$ the Laplacian in $\bR^{d}$.
It\^o and McKean in \cite{IM} characterized $r(x)$ for which
the origin is $\Delta_{d},G_{d}$-regular with $G_{d}=G$
defined in \eqref{4.17.1}.
One sees that for $d=3$ and $d=4$ the classes of $r(x)$ yielding
regularity of the origin for $\Delta_{d},G_{d}$ are quite different. Since
everything is invariant with respect to rotations about the $x_{1}$-axis,
it was natural to figure out  what is the difference between the Laplacians
in $d=3$ and $d=4$ in the coordinates $x_{1},r$ ($r=|x'|$). It turns out that in $d=4$
one has the extra term $(1/r)D_{r}u$, so that the $x'$-component
of the corresponding Markov process is pushed away from the origin
harder than in the case of the  three-dimensional Laplacian.
One   gets such an additional term for $d=3$ as well if one
replaces the (two-dimensional) Laplacian with respect to $x'$
with an operator built from the matrix $a(x')$ with one
 eigenvalue  1
corresponding to the  
 eigenvector  $x'/|x'|$ and the other
 eigenvalue  equal to 2
and corresponding  to the  
 eigenvector  orthogonal to $x'/|x'|$. Then, we
recall that for $\Delta_{4}$ the function $1/|x|^{2}$ is
harmonic and from  the start instead of the 3d Laplacian we
take $L_{\lambda}$ with $\lambda=2$ in $\bR^{3}$. By the
above argument we   conclude that the function
$$ 
u(x_{1},|x'|
)=\int_{0}^{1}\frac{t^{2}}{(t-x_{1})^{2}+|x'|^{2}}\,dt,\quad
x\in\bR^{3},
$$ 
satisfies $L_{\lambda}u=0$ if $x_{1}\not\in[0,1]
$ or $x'\ne0$.  It follows that for any
function
$r(x_{1})$, $x_{1}\in[0,2]$, such that $r(0)=r(2)=0$
and $r(x_{1})>0$ for $x_{1}\in(0,2)$, we have
 $$
u(x_{1},|x'| )=v(x_{1}, |x'|)
$$
in
$$
D_{3}=\bR^{3}\setminus\{x:x_{1}\in[0,2],|x'|\leq r(x_{1})\},
$$
where $v=v(x_{1},\sqrt{x_{2}^{2}+x_{3}^{2}+x_{4}^{2}})$
is the solution of $\Delta_{4}v=0$ in
$$
D_{4}=\bR^{4}\setminus\{x:x_{1}\in[0,2],|x'|\leq r(x_{1})\}
$$
with the boundary data $v(x_{1}, |x'|)=u(x_{1},|x'| )$
on $\partial D_{4}$. By   Section 7.11 in  \cite{IM}, if
$$
 \int_0     
\frac{dx}{x[x/r(x)]}<\infty 
\quad\text{and}\quad
 \int_0 \frac{  dx } {x  |\ln [x/r(x) ]|  } =\infty,
$$
then the origin is regular relative to
$\Delta_{3}, D_{3}$ and is not regular relative
to $\Delta_{4}, D_{4}$, in which case
the limit
$$
\lim_{x_{1}\downarrow0}v(x_{1},r(x_{1}))
=\lim_{x_{1}\downarrow0}u(x_{1},r(x_{1}))
$$
if it exists, most likely is different from $v(0)
=u(0)=1$, and, if this is true, we have a definite proof that
the origin is not regular relative 
to $L_{\lambda},D_{ 3}$    
from the point of view of Markov processes. Observe that
the coefficients of $L_{\lambda}$ are discontinuous 
inside $D_{3}$  and there
is no PDE theory or regular points for operators 
whose coefficients
are discontinuous inside domains.

More generally, we know from Section \ref{section 2} how to build
  an operator in $x'$ variables for $d=3$  whose
radial part is $u_{rr}+(\lambda/r)u_{r}$. It corresponds  
to the Laplacian in, so to speak, $\lambda+1$-dimensional space. When we add
the second order derivative with respect to $x_{1}$ and obtain
 $L_{\lambda}$, we are dealing
with the Laplacian in  $\lambda+2 $-dimensional space, where $1/|x|^{\lambda}$ is a harmonic function. This leads to the guess that
the  direct analogue of
\eqref{4.14.1}:
$$
v_{ \lambda}(x):=\int_{0}^{1}\frac{t^{\lambda}}
{[(t-x_{ 1})^{2}+|x'|^{2}]^{\lambda/2}}\,dt,
\quad x\in\bR^{3},
$$ is an $L_{\lambda}$-harmonic function and if $\lambda>1$, so
that $\lambda+2>3$, then we expect the same conclusions to hold
as in the case of $\lambda=2$.  These were our starting ideas.

One more relevant comment  is that, since the operators $L_{\lambda}$ have
discontinuous coefficients for which there is no PDE theory
of regular points, we avoided using the above, somewhat incomplete, probabilistic arguments in order to attract readers not familiar
with the theory of Markov processes. This led to considering
$G'$ in place of $G$ and introducing  $u$ by
\eqref{4.18.4}
rather  than  using the above $v_{ \lambda}$.

     \end{document}